\documentclass[11pt]{article}

\usepackage{amsmath,mathtools}
\usepackage{amsmath, amsthm, amssymb, amsfonts, enumerate}

\usepackage[colorlinks=true,linkcolor=blue,urlcolor=blue]{hyperref}
\usepackage{dsfont}
\usepackage{geometry}
\usepackage{hyperref}
\usepackage{epstopdf}

\newcommand{\stkout}[1]{\ifmmode\text{\sout{\ensuremath{#1}}}\else\sout{#1}\fi}

\usepackage{hyperref}
\usepackage[english]{babel}
\usepackage{fancyhdr}
\usepackage{fullpage}
\usepackage[dvips]{graphicx}
\usepackage{exercise}
\usepackage{amsmath}
\usepackage{amssymb}
\usepackage{mathrsfs}
\usepackage{amsthm}
\usepackage{array}
\usepackage{enumerate}
\usepackage{mathrsfs}
\usepackage{mathtools}

\usepackage[usenames]{color}
\definecolor{red}{rgb}{1.0,0.0,0.0}

\definecolor{blu}{rgb}{0.0,0.0,1.0}

\definecolor{gre}{rgb}{0.03,0.50,0.03}

\definecolor{darkviolet}{rgb}{0.58, 0.0, 0.83}

\def\eps{\varepsilon}

\DeclareMathOperator{\Hcv}{\mathcal{H}_{\mbox{\footnotesize}{cv}}}
\newtheorem{Theorem}{Theorem}[section]

\newtheorem{Proposition}[Theorem]{Proposition}
\newtheorem{Lemma}[Theorem]{Lemma}

\newtheorem{Remark}[Theorem]{Remark}

\numberwithin{equation}{section}

\def\R{\mathbb R}

\def\E{\mathbb E}

\def\d{\mathrm{d}}
\def\<{\left\langle }
\def\>{\right\rangle }

\def\eps{\varepsilon}

\def\A{{\cal A}}
\def\P{{\cal P}}

\def\R{{\mathbb R}}

\def\E{{{\rm I} \kern -.15em {\rm E}    }}

\newtheoremstyle{mytheorem}
{6pt}
{6pt}
{\itshape}
{-0pt}
{\large \scshape}
{}
{1em}
{}

\newtheoremstyle{myremark}
{6pt}
{10pt}
{\rm}
{-0pt}
{\large \scshape}
{}
{1em}
{}

\def\eps{\varepsilon}

\def\A{\mathcal{A}}

\def\d{\delta}

\def\eps{\varepsilon}

\def\P{\mathfrak{P}}

\def\A{\mathcal{A}}

\def\norm{{\| \kern -.05em | }}

\newcommand{\reals}{{{\rm I} \kern -.15em {\rm R} }}
\newcommand{\nat}{{{\rm I} \kern -.15em {\rm N} }}

\def\R{\mathbb R}

\def\E{\mathbb E}
\def\P{\mathbb P}

\def\F{\mathbb F}

\def\to{\rightarrow}
\def\d{\mathrm{d}}
\def\<{\left\langle }
\def\>{\right\rangle }

\usepackage{color}
\usepackage{latexsym,dsfont}
\usepackage{amssymb}
\usepackage{amsmath}
\usepackage{bbm}
\usepackage{amsthm}
\usepackage{graphicx}
\usepackage{epsfig}
\usepackage{float}
\usepackage{stmaryrd}
\usepackage{fancyhdr}
\usepackage{enumerate}
\usepackage{hyperref}

\def\1{\mathbbm{1}}

\def \R{\mathbb{R}}

\def \E{\mathbb{E}}
\def \F{\mathbb{F}}

\def \P{\mathbb{P}}

\def \eps{\varepsilon}

\def\beqs{\begin{eqnarray*}}
\def\enqs{\end{eqnarray*}}
\def\beq{\begin{eqnarray}}
\def\enq{\end{eqnarray}}

\newcommand{\nc}{\newcommand}
\nc{\esssup}{\mathop{\mathrm{ess\,sup}}}

\newlength{\llegende}
\definecolor{Gcolor}{rgb}{1,0,0.5}

\begin{document}
\title{High risk aversion Merton's problem\\  without transversality conditions}
\author{
Enrico Biffis\footnote{Department of Finance, Imperial Business School, London SW7 2AZ, UK, email: \texttt{e.biffis@imperial.ac.uk}.} \ \ \  \ \ Cristina Di Girolami\footnote{Dipartimento di Matematica, Universit\`a Alma Mater Studiorum Bologna, Piazza di Porta San Donato, 5, 40126 Bologna BO, Italy, email: \texttt{cristina.digirolami2@unibo.it}.} \ \ \ \ \ \ Salvatore Federico\footnote{Dipartimento di Matematica, Universit\`a Alma Mater Studiorum Bologna, Piazza di Porta San Donato, 5, 40126 Bologna BO, Italy, email: \texttt{s.federico@unibo.it}.} \ \ \ \  \ \ Fausto Gozzi\footnote{Dipartimento di Economia e Finanza, Libera Università degli Studi Sociali ``Guido Carli'', Rome, Italy, email:  \texttt{f.gozzi@luiss.it}.}
}

\maketitle

{\bf Abstract} {\footnotesize{{This paper revisits the classical Merton portfolio choice problem over infinite horizon for high risk aversion, addressing technical challenges related to establishing the existence and identification of optimal strategies. Traditional methods rely on perturbation arguments and/or impose restrictive conditions, such as large discount rates and/or bounded strategies, to ensure well-posedness. Our approach leverages the problem's homogeneity to directly solve the associated Hamilton-Jacobi-Bellman equation and verify the optimality of candidate strategies without requiring transversality conditions.}}
}
\medskip




\section{Introduction}
Since its appearance, the seminal contribution of Merton \cite{Merton1969,Merton1971} has continued to stimulate research on lifetime consumption and portfolio choice. While recent work brings machine–learning insights to the classical setting \cite{DDJZ}, a persistent technical challenge in the traditional model concerns the high–risk–aversion case (\(\gamma>1\) in our notation). In this case, giving a complete verification argument is notoriously cumbersome (a detailed overview is provided in \cite{HerHobJe2020MAFI}), and, apart from a few notable exceptions reviewed below, the issue has often been glossed over or circumvented by use of powerful duality methods (see, for example, \cite{KS98}).
	
	The Merton model features a wealth process with linear dynamics whose controlled drift can be, in principle, pushed arbitrarily high by leveraging the risky asset allocation. In addition, even if the agent is prohibited from borrowing or shorting, the multiplicative noise alone can produce large excursions. When $\gamma>1$, the agent's utility can explode to $-\infty$ when consumption falls to zero from above.
	Proving existence of an optimal policy therefore requires ruling out the possibility to postpone consumption to bet on unbounded future payoffs. As a result, any existence proof must simultaneously
	(i) ensure finiteness of the value function (well–posedness), (ii) enforce the intertemporal budget constraint, and (iii) exclude leftover value at infinity (a transversality condition); otherwise, the agent could raise current consumption slightly without violating the budget constraint, but contradicting optimality.
	
	Point i) has been addressed by Merton (\cite{Merton1971}) by requiring that  the agent's subjective discount rate  is ``sufficiently large''  and later by identifying a precise lower bound, which corresponds to condition  \eqref{eq:standing} in this paper. Point ii) is addressed by noting that the Merton problem is associated with a Black and Scholes market for which an explicit local martingale deflator exists, which is a true martingale thanks to Novikov's condition. This is a cornerstone of the martingale approach to the problem pioneered by \cite{CH1989}.
	As for point  iii), the transversality condition requires a vanishing  limit behavior of the  discounted value function, which is notoriously problematic when \(\gamma>1\).
	
	Important contributions to addressing the case of high risk aversion in the Merton problem include \cite{K86},  \cite{DN}, and  \cite{HerHobJe2020MAFI}. \cite{K86} are mainly concerned with the possibility of bankruptcy, and augment the model with a payment the agent can receive at the zero wealth boundary. They recover the standard Merton problem as the payment vanish from above
	under the assumption of positive discounting and a strictly positive riskless rate.
	 \cite{DN}  uses a deterministic perturbation of the value function to recover the candidate solution of the Merton problem as the perturbation vanish, but they restrict the investment strategy to bounded portfolio weights.
	 More recently,  \cite{HerHobJe2020MAFI} offer a self-contained treatment of the Merton problem for all parameter configurations ensuring well-posedness. They use again a perturbation approach, this time using the candidate optimal
	 consumption stream to scale the associated wealth process. They also  work within a subset of admissible strategies for which the transversality condition holds ({\it fiat} conditions in their language). Their Theorem~5.1 and Corollary~5.4 are the counterparts of our results.
	
	 Our method provides a concise, self-contained proof of existence and optimality for the Merton problem when \(\gamma > 1\), under the standard assumption that the agent's discount  rate satisfies the lower bound given by \eqref{eq:standing}.
	Instead of analyzing perturbations of the original problem, as done in \cite{K86,DN,HerHobJe2020MAFI}, we adopt a more direct approach that avoids the difficulties associated with verifying transversality conditions.
More specifically, leveraging the problem's homogeneity, we proceed through the following steps\footnote{Our approach rigorously develops, makes precise, and simplifies the arguments sketched in \cite{biffis2020optimal} in a different setting including labor income.}:
\begin{enumerate}
    \item First, we demonstrate that the value function is finite, nontrivial, and homogeneous (Subsection \ref{subsec:finiteness}).
    \item Next, we prove that the value function solves the associated Hamilton-Jacobi-Bellman (HJB) equation, which allows us to explicitly derive its form (Subsection \ref{sec:HJB}).
    \item Finally, we verify that the candidate optimal feedback map indeed yields an optimal strategy (what we call ``half-verification''). This step turns out to be straightforward and does not require checking any transversality conditions  (Subsection \ref{sub:ver}).
\end{enumerate}
Subsection \ref{remTrCon} provides an overview of the method and compares it with the classical verification approach.

\section{Formulation of the stochastic optimal control problem}
Let $W$ be a standard one dimensional Brownian motion defined on a filtered probability space $(\Omega,\mathcal{F},\F:=(\mathcal{F}_{t})_{t\geq 0}, \P)$ satisfying the usual conditions.
We assume that an investor may continuously invest, over the time interval $[0,\infty)$, in a money market account, with market value $(S^o_{t})_{t\geq 0}$, and a risky asset, with market value denoted  by $(S_{t})_{t\geq 0}$. The market model is the standard Black-Scholes model.
 The money market account has deterministic dynamics
$$
\d S^o_{t}=rS^o_{t}\d t, \ \ \ S^o_{0}=1,
$$
with $r\in\R$, whereas the risky asset dynamics obeys the SDE
$$
\d S_{t}=\mu S_{t}\d t + \sigma S_{t}\d W_t, \ \ \ S_{0}=s_{0}>0,
$$
with $\mu\in\R, \,\sigma>0$.  Hence, $(S_{t})_{t\geq 0}$ is a geometric Brownian motion and we can write
$$
S^o_{t}=e^{rt}, \ \ \ \ S_{t}=s_{0}\exp\left[\left(\mu-\frac{\sigma^2}{2}\right)t+\sigma W_{t}\right].
$$
Introducing the notation $\displaystyle{\lambda:=\frac{\mu-r}{\sigma}}$ for the market risk premium, we can also write
$$
\d S_{t}=(r+\sigma\lambda)\, S_{t}\d t + \sigma S_{t}\d W_{t}, \ \ \ S_{0}=s_{0}>0.
$$

Let us denote by $L^{2,\F}_{\mathbf{loc}} (\Omega\times [0,\infty);\mathbb{R})$ and $L^{1,\F}_{\mathbf{loc}} (\Omega\times [0,\infty);\mathbb{R}^{+})$ the spaces  of $\F$-adapted processes $(Z_{t})_{t\geq 0}$ valued in $\R$ and $\R^{+}$, respectively, where the latter denotes the set of nonnegative real numbers,  and  such that, $\E\left[\int_{0}^{R}|Z_{t}|^{2}\d t\right]<\infty$ and  $\E\left[\int_{0}^{R}|Z_{t}|\d t\right]<\infty$, respectively, for every $R>0$.

An investor endowed with initial wealth $x>0$ can dynamically rebalance her wealth and support a consumption flow. In particular, let us introduce the following stochastic processes:
\begin{enumerate}[(i)]
\item $(X_t)_{t\geq 0}$, representing the market value of the agent's financial wealth over time;
\item $(\pi_{t})_{t\geq 0}\in L^{2,\F}_{\mathbf{loc}} (\Omega\times [0,\infty);\mathbb{R})$, representing the investment strategy of the investor expressed in terms of \emph{amount of money} allocated to the risky asset, whereby a negative allocation is interpreted as short selling;
\item  $(c_{t})_{t\geq 0}\in L^{1,\F}_{\mathbf{loc}} (\Omega\times [0,\infty);\mathbb{R}^{+})$, representing  the nonnegative (rate of)  consumption out of financial wealth.
\end{enumerate}
The above processes will also be denoted simply by $X,\pi,c$ when no confusion arises.

Assuming that the portfolio is self-financing, in the sense that there are no capital injections and that consumption is the only source of withdraws, the wealth process obeys the following SDE:
\begin{align}\label{stateMerton}
\d X_{t}&=\pi_{t} \frac{\d S_{t}}{S_t}+(X_t-\pi_{t}) \frac{\d S^o_t}{S^o_{t}}-c_t\d t\nonumber \\&=\pi_{t} \left((r+\sigma\lambda)\d t+\sigma\d W_{t}\right)+(X_{t}-\pi_{t}) \,r \d t-c_t\d t\\
&= rX_t dt +\sigma\lambda \pi_{t}\d t -c_t\d t +\sigma \pi_{t}\d W_{t}.\nonumber
\end{align}
The state equation governing the wealth dynamics is therefore
\begin{equation}	\label{eqCSE}
\begin{cases}
\d X_t = (r X_t + \sigma\lambda\pi_t)   \, \d t + \sigma \pi_t   \, \d W_t - c_t \, \d t,\\
X_{0}=x>0,
\end{cases}
\end{equation}
where $X$ is the state variable of the stochastic control problem we are going to define, whereas  $(c,\pi)$ denotes the pair of control variables.  We denote by
 $(X_t^{x; c,\pi})_{t\geq 0}$,
or simply by
$X$
if no confusion arises,
the unique  strong solution of the controlled SDE introduced above. The set of admissible strategies considered is defined as:
$$\mathcal{A}(x)=\Big\{(c, \pi)\in  L^{1,\F}_{\mathbf{loc}} (\Omega\times [0,\infty);\mathbb{R}^{+})  \times L^{2,\F}_{\mathbf{loc}} (\Omega\times [0,\infty);\mathbb{R}): \  X\geq 0\Big\}.$$

Given  $\rho\in\mathbb R$, the objective functional to be maximized over the set $\mathcal{A}(x)$ is
$$
J(x;c,\pi)=J (c)=  \mathbb{E} \left[ \int_0^\infty e^{-\rho s} \frac{c_s^{1-\gamma}}{1-\gamma} \, \d s \right],
$$
where we focus on the case of high risk aversion $\gamma>1$, which is  more problematic and often neglected in the literature (see discussions in \cite{biffis2020optimal} and \cite{HerHobJe2020MAFI}).
We assume
\begin{equation}\label{eq:standing}
	\rho \;>\; (1-\gamma)\!\left(r+\frac{\lambda^2}{2\gamma}\right),
\end{equation}
which is the standard Merton finiteness condition for $\gamma>1$; see \cite[equation~(20)]{biffis2020optimal}, \cite[equation~(5)]{HerHobJe2020MAFI}, and Section~\ref{subsec:finiteness} for further discussion.
Clearly, $J(c)$ is well defined for every $c\in L^{1,\F}_{\mathbf{loc}} (\Omega\times [0,\infty);\mathbb{R}^{+})$, is
nonpositive and possibly equal to $-\infty$.
We define the value function
 \begin{equation}		\label{eqValueFunc}
	V(x) := \sup_{(c,\pi)\in\mathcal{A}(x)}J(x;c,\pi).
\end{equation}
As a consequence of the nonpositivity of $J$, we have  $V\in[-\infty,0]$.
 A consumption-investment strategy
 $(\widehat{c},\widehat {\pi})\in \mathcal{A}(x)$ starting at $x>0$ is said \emph{optimal} if
$$
 V(x) =\mathbb{E} \left[ \int_0^\infty e^{-\rho s} \frac{\widehat{c_{s}}^{1-\gamma}}{1-\gamma} \, \d s \right].
$$
%
%
%
%

\section{Solution}
We provide the solution of our stochastic optimal control problem basically  merging a direct approach on the Hamilton-Jacobi-Bellman equation and a verification approach to prove the optimality of the candidate optimal feedback map.

\subsection{Estimates on admissible consumption plans}
In our Black-Scholes setting, the market is complete and, by no arbitrage,  there exists a unique measure equivalent to $\mathbb{P}$ such that $(S^0,S)$ is a local martingale. Its density, denoted by $Y$, is called local martingale deflator and is the unique strictly positive $\mathbb F$–adapted local martingale $Y$ with $Y_0=1$ and such that $YS^0$ and $YS$ are $\mathbb P$–local martingales. Such process $Y$ is explicitly written as
$$
Y_t=\exp(-rt-\lambda W_t-\tfrac12\lambda^2 t).
$$
The following Lemma then provides a useful estimate on admissible consumption plans.

\begin{Lemma}[Budget constraint via true martingale deflator]\label{lem:budget-Y}
	For any admissible control $(c,\pi)\in\A(x)$ with wealth $X$ and any $\F$-stopping time $\tau$ (not necessarily bounded), we have
	\begin{equation}\label{eq:budget-Y}
		\E\!\left[\int_{0}^{\tau} Y_s\,c_s\,\d s\right]\;\le\; x.
	\end{equation}
\end{Lemma}

\begin{proof}
{
	Let $\varphi^0_t := \frac{X_t-\pi_t}{S^0_t}$ and $\varphi^1_t := \frac{\pi_t}{S_t}$, so that $X_t=\varphi^0_tS^0_t + \varphi^1_t S_t$ and
	$\d X_t=\varphi^0_t\,\d S^0_t+\varphi^1_t\,\d S_t - c_t\,\d t$, as the strategy is self-financing. Define
	\[
	M_t:=Y_tX_t+\int_0^t Y_s\,c_s\,\d s.
	\]
	By straightforward computations, we can write:
	\[	\d(Y_tX_t)=\varphi^0_t\,\d(Y_tS^0_t)+\varphi^1_t\,\d(Y_tS_t)-Y_t c_t\,\d t.
	\]
Hence, we have
$$
\d M_t=\varphi^0_t\,\d(Y_tS^0_t)+\varphi^1_t\,\d(Y_tS_t).
$$
As observed, in our Black--Scholes setup, Novikov condition holds and both $Y_tS^0_t$ and $Y_tS_t$ are $\mathbb P$-martingales; therefore $M$ is a local martingale with $M_0=x$, and $M_t\ge0$ since $Y,X,c\ge0$.
}
	To obtain the inequality, let $(\sigma_n)$ denote a sequence of stopping times localizing $M$. Optional sampling then gives
	\[
	\E\!\left[M_{\tau\wedge\sigma_n}\right]=\E[M_0]=x.
	\]
	Dropping $\E[Y_{\tau\wedge\sigma_n}X_{\tau\wedge\sigma_n}]\ge0$ and using monotone convergence as $n\uparrow\infty$ finally yields \eqref{eq:budget-Y}.
\end{proof}

\subsection{Finiteness and homogeneity of $V$}\label{subsec:finiteness}
In this subsection we prove, beforehand, some preliminary properties of the value function.
\begin{Proposition}		\label{VfinitaOm}
There exists $a>0$ such that
\begin{equation}   	\label{equa}
V(x)=a\dfrac{x^{1-\gamma}}{1-\gamma}.
\end{equation}
\end{Proposition}
\begin{proof}
\begin{enumerate}	
\item Finiteness and non–triviality under \eqref{eq:standing}.
\begin{itemize}
\item \emph{Finiteness.}
Pick any $\kappa\in(0,\kappa_0)$, where the threshold  $\kappa_0 $  is obtained from condition  \eqref{eq:standing}, yielding:
\[
\kappa_0 \;:=\; \frac{\rho}{\gamma-1} - r - \frac{\lambda^2}{2\gamma}\;>\;0.
\]
Consider the proportional feedback strategy defined as
\[
\pi_t \;=\; \frac{\lambda}{\sigma\gamma}\,X_t,\qquad c_t \;=\; \kappa\,X_t .
\]
Then corresponding wealth equation $X$ is a GBM solving the SDE
\[
\frac{\d X_t}{X_t} \;=\; \Big(r+\frac{\lambda^2}{\gamma}-\kappa\Big)\d t \;+\; \frac{\lambda}{\gamma}\,\d W_t.
\]
Hence, for $\gamma>1$, we have
{
\[
\E\!\left[e^{-\rho s}\,\frac{c_s^{\,1-\gamma}}{1-\gamma} \right]
= \frac{\kappa^{\,1-\gamma}}{1-\gamma}\,x^{1-\gamma}
\exp\!\Big(\big[(1-\gamma)\Big(r-\kappa+\frac{\lambda^2}{2\gamma}\Big)-\rho\big]\,s\Big).
\]}
As $\kappa\in(0,\kappa_0)$, the exponent in the exponential above is negative, yielding
\[
J(c)\;=\;\E\!\left[\int_0^\infty e^{-\rho s}\,\frac{c_s^{\,1-\gamma}}{1-\gamma}\,\d s\right]\;>\;-\infty,
\]
which proves $V(x)>-\infty$.
\item \emph{Non–triviality.}
By contradiction, assume that $V(x)\equiv0$ on $(0,\infty)$. Since $u(c)=\frac{c^{1-\gamma}}{1-\gamma}\le0$,
there exists a sequence $(c^{(n)},\pi^{(n)})\in A(x)$ with
\[
0\;\ge\;\lim_{n\to\infty}\E\!\left[\int_0^1 e^{-\rho s}\,\frac{(c^{(n)}_s)^{1-\gamma}}{1-\gamma}\,\d s\right]=0.
\]
Hence, along a subsequence $(c^{(n)}_s)^{1-\gamma}\to0$ a.e. on $\Omega\times[0,1]$, which implies $c^{(n)}_s\to\infty$ a.e.
By Lemma~3.1 (budget constraint with deflator $Y$) with $\tau\equiv1$, we have
\[
\E\!\left[\int_0^1 Y_s\,c^{(n)}_s\,\d s\right]\;\le\;x\quad\text{for all }n.
\]
Using $Y_s>0$ a.e. and Fatou’s lemma, we can write
\[
+\infty \;=\; \E\!\left[\int_0^1 \liminf_{n\to\infty}Y_s\,c^{(n)}_s\,\d s\right]
\;\le\;\liminf_{n\to\infty}\E\!\left[\int_0^1 Y_s\,c^{(n)}_s\,\d s\right]\;\le\;x,
\]
which gives a contradiction. Therefore $V$ is not identically zero on $(0,\infty)$.
\end{itemize}

\item $(1-\gamma)-$homogeneity.

Fix $\alpha>0$. Given  $(c,\pi)\in \A(x)$, set
$(\tilde c,\tilde\pi):=(\alpha c,\alpha \pi)$. {Clearly, $(\tilde{c},\tilde{\pi})\in  L^{1,\F}_{\mathbf{loc}} (\Omega\times [0,\infty);\mathbb{R}^{+})  \times L^{2,\F}_{\mathbf{loc}} (\Omega\times [0,\infty);\mathbb{R})$}. Moreover,
 by linearity of the
state equation, the unique strong solution satisfies
$X^{\alpha x;\tilde c,\tilde\pi}_t=\alpha\,X^{x;c,\pi}_t$ for all $t\ge0$.
Hence $X^{x;c,\pi}\ge0$ implies $X^{\alpha x;\tilde c,\tilde\pi}\ge0$, so that
$(\tilde c,\tilde\pi)\in \A(\alpha x)$. Conversely, if
$(\hat c_,\hat\pi)\in \A(\alpha x)$ then
$(c,\pi):=(\alpha^{-1}\hat c,\alpha^{-1}\hat\pi)\in \A(x)$.
Therefore $\alpha\,\A(x)=\A(\alpha x)$.


\medskip
Now,  by the $(1-\gamma)$–homogeneity
of $u(c)=\frac{c^{1-\gamma}}{1-\gamma}$,  we have  $J(\alpha c)=\alpha^{\,1-\gamma}J(c)$. Therefore, we obtain
\[
V(\alpha x)=\sup_{(\tilde c,\tilde\pi)\in \A(\alpha x)}J(\tilde c)
=\sup_{(c,\pi)\in \A(x)}J(\alpha c)=\alpha^{\,1-\gamma}V(x),\qquad \alpha>0.
\]
Letting $a:=-\,(1-\gamma)\,V(1)\in(0,\infty)$ (by item~1, $V$ is not identically $0$ and
$V\le0$), we get
\[
V(x)=a\, \frac{x^{\,1-\gamma}}{1-\gamma},
\]
which proves the claim.

\vspace{-.65cm}

\end{enumerate}
\hspace{5cm}
\end{proof}

\subsection{The value function as solution to the Hamilton-Jacobi-Bellman equation}\label{sec:HJB}
Standard arguments of stochastic control (see, e.g., \cite[Ch.\,3]{PhamBook} or \cite[Ch.\,4]{YongZhou}) lead to associate to the value function $V$  the stationary Hamilton-Jacobi-Bellman (HJB) equation:
\begin{equation}\label{HJBMerton3}
   \rho  v(x) =\mathcal{H}_{\max}(x,v'(x),v''(x)), \ \ \ \ x>0,\end{equation}
   where
\begin{equation}\label{maxH}\mathcal{H}_{\max}(x,p,P)= \sup_{(c,\pi)\in \mathbb{R}_{+}\times \mathbb{R} }\Hcv (x,p,P;c,\pi),  \ \ \ x>0, \  \ (p,P)\in \R^{2}, \ \ (c,\pi)\in \R^+\times \R
\end{equation}
and
$$
\Hcv (x,p,P; c,\pi)=rxp + \pi \sigma \lambda p+ \frac{1}{2} \pi^2 \sigma^2  P-cp +\frac{c^{1-\gamma}}{1-\gamma}, \ \ \  x>0, \  \ (p,P)\in \R^{2}.
$$
Note that, when  $p>0,\ P<0$, the unique maximum point of the function
$$\mathbb{R}_{+}\times \mathbb{R}\to\mathbb{R},\ \  \ (c,\pi)\mapsto \Hcv(x,p,P;c,\pi)$$ is provided by
\begin{equation}\label{maxcpi}
c_{\max}(x,p,P)= p^{-1/\gamma} , \quad \pi_{\max} (x,p,P)=-\frac{\lambda p}{\sigma  P} .
\end{equation}
Thus,  in this case,
\begin{equation}\label{eq318b}
\mathcal{H}_{\max}(x,p,P)=\Hcv (x,p,P; c_{\max},\pi_{\max})
\end{equation} so
$$\mathcal{H}_{\max}(x,p,P)=rxp-\frac{\lambda^{2}}{2}\frac{p^{2}}{P}+ \frac{\gamma}{1-\gamma }p^{\frac{\gamma -1}{\gamma}}, \ \ \ x>0, \  \ p>0,  \quad P<0.$$
Therefore, considering that  by Proposition \ref{VfinitaOm} we have $V'>0$ and $V''<0$, the value function $V$ is expected to solve  the HJB equation
\begin{equation}\label{HJB3}
\rho v(x)= rxv'(x)-\frac{\lambda^{2}}{2}\frac{v'(x)^{2}}{v''(x)}+ \frac{\gamma}{1-\gamma }v'(x)^{\frac{\gamma -1}{\gamma}}.
\end{equation}

The connection of $V$ with \eqref{HJB3} passes through the Dynamic Programming Principle, which reads as
\begin{equation}\label{DPP}
V\left({x}\right) =\sup_{(c,\pi)\in {\cal A}\left({x}\right) }\E\left[
\int_{0}^{ \tau_{(c,\pi)}} e^{-\rho s}\frac{c_{s}^{1-\gamma}}{1-\gamma}\d s +
e^{-\rho \tau_{(c,\pi)}}V (X_{\tau_{(c,\pi)}}) \right],
\end{equation}
for every family of  stopping times $\left(\tau_{(c,\pi)}\right)_{(c,\pi)\in {\cal A}\left({x}\right)}$.
This is a well-established equation in the general theory of stochastic optimal control, at least when the value function is known to be continuous (see \cite[Ch.\,4]{YongZhou} in the case when $\tau_{(c,\pi)}\equiv \hat t$ deterministic; \cite[Ch.\,3]{PhamBook} or \cite[Ch.\,2]{Touzibook} for our case).

\begin{Proposition}		\label{SolvesHJB}
$V$ solves \emph{HJB} \eqref{HJB3} in classical sense at each $x>0$.
\end{Proposition}

\begin{proof}
The proof of the fact that $V$ is a supersolution at each $x>0$, that is that the inequality $\geq$ holds in \eqref{HJB3} with $v=V$  is standard; we refer for instance to \cite[Section 4.3]{PhamBook}.

The proof of the fact that $V$ is a subsolution is less standard due to the unboundedness of the set of controls, i.e. $\R$ for $\pi$ and $\R^{+}$ for $c$. We therefore provide it.

Consider the function
\begin{eqnarray*}
\varphi (y) = V(y) + |y-x|^3,  \;\;\;\;\;\forall y \in (0,\infty).
\end{eqnarray*}
We observe that
\begin{equation}\label{pops}
\varphi(x) = V(x),  \ \ \varphi'(x) = V'(x),  \ \ \varphi''(x) = V''(x), \ \ \mbox{and} \ \  \varphi \geq V.
\end{equation}

Assume, by contradiction that the strict inequality $>$ holds in \eqref{HJB3} with $v=V$ at some point $x>0$. Then, by \eqref{pops},  it also holds the strict inequality $>$ in \eqref{HJB3} with $v=\varphi$ at $x>0$.
By continuity, there exist $\delta, \eps>0$ such that
\begin{eqnarray}\label{contradictionCons3}
\rho \varphi(y) & \geq &  \delta + ry\varphi'(y)-\frac{\lambda^{2}}{2}\frac{\varphi'(y)^{2}}{\varphi''(y)}+ \frac{\gamma}{1-\gamma }\varphi'(y)^{\frac{\gamma -1}{\gamma}}\nonumber\\
&= & \delta + \mathcal{H}_{\max}(y,\varphi'(y),\varphi''(y))
\\&\geq & \delta + \Hcv(y,\varphi'(y),\varphi''(y);c,\pi) \ \ \ \ \ \ \ \ \ \ \  \ \ \forall y \in [x-\eps, x+\eps], \ (c,\pi)\in \R^{+}\times \R. \nonumber
\end{eqnarray}
Now, given $(c ,\pi)\in \mathcal A(x)$  and $h \in (0,1)$, set $X_\cdot = X_\cdot^{x;c,\pi}$ and
$$
\tau^h_{c,\pi} := \inf \big\{s \ge 0 \::\: |X_s -x| \geq \eps\big\}  \wedge  h.
$$
Note that $\tau^h_{c,\pi}$ is a stopping time strictly larger than $0$. By definition of $\tau^h_{c,\pi}$, we have
$X_{s} \in [x-\eps, x+\eps]$ for each $s \in [0,\tau^h_{c,\pi}]$. Then, using \eqref{contradictionCons3}, we get
\begin{eqnarray}
\rho \varphi(X_s) \geq
  \delta + \Hcv\left(X_s,\varphi'(X_{s}),\varphi''(X_{s});\,c_s,\pi_{s}\right)
\ \  \forall  s \in [0,\tau^h_{c,\pi}]. \label{contradictionCons4}\end{eqnarray}
We apply It\^o's formula to $e^{-\rho s}  \varphi(X_s)$ in the interval $[0, \tau^h_{c,\pi}]$, obtaining
\begin{eqnarray*}
 \varphi(x) -  e^{-\rho \tau^h_{c, \pi}} \varphi(X_{\tau^h_{c,\pi }})   &=&
\int_0^{\tau^h_{c,\pi}} e^{-\rho s} \left(\rho \varphi({X}_s) -\Hcv(X_s,\varphi'(X_s),\varphi''(X_s);\,c_s,\pi_{s})+\frac{c_{s}^{1-\gamma}}{1-\gamma}\right) \d s\\
&& - \int_0^{ \tau^h_{c,\pi}}\sigma e^{-\rho s}   \varphi' ({X}_s) \pi_{s} \d W_s.
\end{eqnarray*}
We may pass to the expected value by taking account that, due to definition of  $\tau^h_{c,\pi}$, the expected value of the stochastic integral vanishes. We therefore obtain
\begin{eqnarray*}
 \varphi(x) - \E\left[ e^{-\rho \tau^h_{c, \pi}} \varphi(X_{\tau^h_{c,\pi}}) \right]  \ = \ \E\left[
\int_0^{\tau^h_{c,\pi}} e^{-\rho s} \left(\rho \varphi({X}_s) -\Hcv(X_s,\varphi'(X_s),\varphi''(X_s);\,c_s,\pi_{s})+\frac{c_{s}^{1-\gamma}}{1-\gamma}\right) \d s\right].
\end{eqnarray*}
Combining with  \eqref{contradictionCons4}, we get
\begin{eqnarray*}
\varphi(x) - \E \left[ e^{-\rho \tau^h_{c,\pi}} \varphi(X_{\tau^h_{c,\pi}}) \right] \ge
\delta   \E\left[   \int_0^{\tau^h_{c,\pi}} e^{-\rho s}  ds  \right]
 + \E \left[ \int_0^{\tau^h_{c,\pi}}e^{-\rho s} \frac{c_s^{1-\gamma}}{1-\gamma} \d s\right].
\end{eqnarray*}
Since $ \varphi(x) = V(x)$ and $\varphi(y) = V(y)  + |y-x|^3$ for each $y \in (0,\infty)$, recalling that $\tau^h_{c,\pi} \le h <1$, we get
\begin{eqnarray*}
V(x) - \E \left[ e^{-\rho \tau^h_{c,\pi}}V(X_{\tau^h_{c,\pi}}) \right] \ge
\delta e^{-\rho} \E\left[\tau^h_{c_\cdot,\pi_{\cdot}}  +|X_{\tau^h_{c,\pi}}-x|^3 \right]+ \E \left[ \int_0^{\tau^h_{c,\pi}}e^{-\rho s} \frac{c_s^{1-\gamma}}{1-\gamma} \d s\right]
\end{eqnarray*}
i.e.
\begin{eqnarray*}
V(x) -  e^{-\rho} \E\left[\delta\tau^h_{c,\pi}  +|X_{\tau^h_{c,\pi}}-x|^3 \right]\geq \E \left[ e^{-\rho \tau^h_{c,\pi}}V(X_{\tau^h_{c,\pi}}) +\int_0^{\tau^h_{c,\pi}}e^{-\rho s} \frac{c_s^{1-\gamma}}{1-\gamma} \d s\right].
\end{eqnarray*}
Taking ${\sup}_{(c,\pi) \in \A(x)}$ on both sides of the previous inequality and recalling  DPP \eqref{DPP}, we end up with
$$
\sup_{(c,\pi) \in \A(x)}\left(-   \E\left[\delta\tau^h_{c,\pi}  +|X_{\tau^h_{c,\pi}}-x|^3 \right]\right)\geq 0,
$$
i.e.
\begin{eqnarray}\label{lastIneq}
\inf_{(c,\pi)\in\mathcal {A}(x)}  \E\left[\delta\tau^h_{c,\pi}  +|X_{\tau^h_{c,\pi}}-x|^3 \right]\le 0.
\end{eqnarray}
Now note that, uniformly in $({c},\pi)\in\mathcal{A}(x)$,  we have
\begin{eqnarray}\label{lastIneq1}
\E[\tau^h_{c,\pi}] = \E\left[ \tau^h_{c,\pi} \mathds{1}_{\{\tau^h_{c,\pi} < h\}} + \tau^h_{c,\pi} \mathds{1}_{\{\tau^h_{c,\pi} \ge h\}}\right]
\ge h\,\P\{\tau^h_{c,\pi} \ge h\}
\end{eqnarray}
and
\begin{eqnarray}\label{lastIneq2}
\E \left[ |X_{\tau^h_{c,\pi}}-x|^3 \right] &=&
\E \left[ |X_{\tau^h_{c,\pi}}-x|^3 \mathds{1}_{\{\tau^h_{c,\pi} < h\}} + \ |X_{\tau^h_{c,\pi}}-x|^3 \mathds{1}_{\{\tau^h_{c,\pi} \ge h\}}\right] \nonumber\\
&\geq & \eps^3\,\, \P\{\tau^h_{c,\pi} < h\}
\ \ge \ \eps^3 \ h \ \P\{\tau^h_{c,\pi} < h\}.
\end{eqnarray}
Combining
\eqref{lastIneq1} and \eqref{lastIneq2}, we get, uniformly in $({c},\pi)\in\mathcal{A}(x)$,
$$\delta \E\left[\tau^h_{c,\pi}\right] +  \E \left[ |X_{\tau^h_{c,\pi}}-x|^3 \right]\geq \min\left\{\delta,\eps^{3}\right\}\cdot h,$$
 contradicting \eqref{lastIneq} and concluding the proof.
\end{proof}

\begin{Proposition}	\label{expression a}
The constant $a$ in \eqref{equa} is
\begin{equation}		\label{Equa}
a=\left( \dfrac{\rho-(1-\gamma)\left( r+\dfrac{\lambda^2}{2\gamma} \right)}{\gamma}\right)^{-\gamma}.
\end{equation}
\end{Proposition}
\begin{proof}
By Propositions \ref{VfinitaOm} and \ref{SolvesHJB} we  have
$$
\rho \dfrac{a}{1-\gamma}x^{1-\gamma}=rax^{1-\gamma}+\dfrac{\lambda^2 a}{2\gamma}x^{1-\gamma}+\dfrac{\gamma}{1-\gamma} a^{\frac{\gamma-1}{\gamma}} x^{1-\gamma}.
$$
Since $a>0$ and the above equality must hold for each $x>0$,
we get the equality
$$
\rho\,\dfrac{1}{1-\gamma}=r+\dfrac{\lambda^2}{2\gamma}+\dfrac{\gamma}{1-\gamma} a^{-\frac{1}{\gamma}},
$$
and the claim follows.
\end{proof}

\begin{Remark}
\label{rm:uniq}
Note that the \emph{HJB} equation \eqref{HJBMerton3} does not have a unique solution if we do not add any boundary condition: indeed, a simple check shows that also $v\equiv 0$ is a solution. However, this fact does not create any problems in our approach, as we already know that the value function is not zero (Proposition \ref{VfinitaOm}, \emph{nontrivially part}), which is a key point of our approach.  
\end{Remark}

Finally observe that the parameters appearing in \eqref{Equa} play an important role in shaping the optimal controls, as discussed in the next section.

\subsection{The ``half-verification'' and the optimal feedback strategy}\label{sub:ver}
In this subsection we finalize our approach by providing the true optimality of the candidate optimal feedback map that comes out from the optimization in $\Hcv$.

Using   \eqref{equa}, the candidate optimal feedback map provided by the maximization in \eqref{maxH} (see \eqref{maxcpi})  is the map $G:(0,\infty) \to \R^+\times \R$ defined by
$$
G(x)=(G^c(x),G^\pi(x))=\left(c_{\max}(x,V'(x),V''(x)),\,\,\pi_{\max}(x,V'(x),V''(x)) \right)
=\left( a^{-1/\gamma}x, \,\,\dfrac{\lambda}{\sigma\gamma}x\right),
$$
where $a$ is given in \eqref{Equa}.
Plugging this map in the  state equation \eqref{eqCSE}, we get  the following closed loop equation associated to $G$
\begin{equation}	\label{eqCLE}
\begin{cases}
\d X_t = \left(r+ \dfrac{\lambda^2}{\gamma} -a^{-1/\gamma}\right) X_t  \, \d t + \dfrac{\lambda}{\gamma} X_t\, \d W_t ,\\
X_{0}=x>0,
\end{cases}
\end{equation}

whose  explicit solution is
\begin{equation}	\label{solCLE}
\widehat{X}_t=x\exp{\left[  \left(r+ \dfrac{\lambda^2}{\gamma} -a^{-1/\gamma}-\dfrac{\lambda^2}{2\gamma^2}\right)  t + \dfrac{\lambda}{\gamma} W_t\right]}=
x\exp{\left[  \left(\dfrac{r-\rho}{\gamma}+\dfrac{\lambda^2}{2\gamma}\right)t + \dfrac{\lambda}{\gamma} W_t\right]}.
\end{equation}

Define now the feedback control
$\left(\widehat{c}, \widehat{\pi} \right)$ by
\begin{equation}		\label{optimalStr}
\widehat{c}_t:=c_{\max}(\widehat{X}_t,V'(\widehat{X}_t),V''(\widehat{X}_t))=a^{-1/\gamma} \widehat{X}_t, \quad \widehat{\pi}_t:=\pi_{\max}(\widehat{X}_t,V'(\widehat{X}_t),V''(\widehat{X}_t)) =\dfrac{\lambda}{\sigma \gamma}\widehat{X}_t.
\end{equation}

In line with \cite[Remark~2.2]{DN}, the constant $a$  appearing in \eqref{Equa} shapes the agent's optimal investment strategy, by allowing for hedging (positive allocation across the risky and riskless asset), leverage (riskless borrowing to increase the risky asset allocation), and short selling (increasing the riskless asset allocation by selling short the risky asset).
Differently from  \cite{DN}, and in line with \cite{HerHobJe2020MAFI}, we allow the subjective discount rate $\rho$ to take nonpositive values. As discussed in
\cite{HerHobJe2020MAFI}, Appendix~D, this offers greater flexibility in allowing our problem to encompass different accounting units.

\begin{Theorem}
The feedback strategy  $\left(\widehat{c},\widehat{\pi}\right)$ defined in \eqref{optimalStr} belongs to $\A(x)$  and is  optimal.
\end{Theorem}

\begin{proof}
\begin{enumerate}
\item We first show that $(\widehat{c},\widehat{\pi})\in\mathcal{A}(x)$.
In fact, for all $R<0$,\footnote{Hereafter, recall that $\E\left[ e^{\alpha W_{t}}\right]=e^{\alpha^2t/2}$ for  $\alpha\in\R$.}
\[
\E\left[ \int_0^R |\widehat{c}_t| \d t \right]=a^{-1/\gamma} \E\left[ \int_0^R |\widehat{X}_t |\d t \right] <\infty
\]
and similarly
\[
\E\left[ \int_0^R |\widehat{\pi}_t|^2 \d t \right]=\left(\dfrac{\lambda}{\sigma\gamma}\right)^2  \E\left[ \int_0^R |\widehat{X}_t| ^2 \d t \right]<\infty.
\]
Moreover, since
$
\widehat{X}_{\cdot}$ and $X^{x;\widehat{c},\widehat{\pi}}_{\cdot}$ are indistinguishable processes (by uniqueness of solution of the SDE), we have $X_{\cdot}^{x;\widehat{c},\widehat{\pi}}= \widehat{X}_{\cdot}\geq 0$.

\item
Since  $V$ is a classical solution of the HJB \eqref{HJB3},  by applying It\^o formula to
$e^{-\rho t}V(\widehat{X}_t)$ from $0$ to $t$ we get

\begin{equation*}		
\begin{aligned}
V(x) & =e^{-\rho t} V(\widehat{X}_t)+\int_0^t e^{-\rho s} \frac{\widehat{c}_s\,^{1-\gamma}}{1-\gamma} \d s -\int_0^t   e^{-\rho s} V'(\widehat{X}_s) \sigma\widehat\pi_s \d W_s\\
& +\int_0^t e^{-\rho s}\left[  \rho V(\widehat{X}_s) -\Hcv\left(\widehat{X}_s, V^{\prime}(\widehat{X}_s), V^{\prime \prime}(\widehat{X}_s) ;\widehat{c}_s, \widehat{\pi}_s \right) \right] \d s \\
&=e^{-\rho t} V(\widehat{X}_t)+\int_0^t e^{-\rho s} \frac{\widehat{c}_s\,^{1-\gamma}}{1-\gamma} \d s -\int_0^t   e^{-\rho s} V'(\widehat{X}_s) \sigma\widehat\pi_s \d W_s\\
& +\int_0^t e^{-\rho s}\left[\mathcal{H}_{\max }\left(\widehat{X}_s, V^{\prime}(\widehat{X}_s), V^{\prime \prime}(\widehat{X}_s)\right) -\Hcv\left(\widehat{X}_s, V^{\prime}(\widehat{X}_s), V^{\prime \prime}(\widehat{X}_s) ;\widehat{c}_s, \widehat{\pi}_s \right) \right] \d s .
\end{aligned}
\end{equation*}
By definition of $(\widehat{c},\widehat{\pi})$, the integrand in the last integral vanishes. So, the previous equality rewrites as
\begin{equation}		\label{eq322}
V(x)=e^{-\rho t} V(\widehat{X}_t)+\int_0^t e^{-\rho s} \frac{\widehat{c}_s\,^{1-\gamma}}{1-\gamma} \d s -   \int_0^t   e^{-\rho s} V'(\widehat{X}_s)\sigma \widehat\pi_s \d W_s.
\end{equation}

Now, observe that the stochastic integral is a zero-mean martingale; in fact, by the expression for $V$ provided by \eqref{equa}, we have
\begin{align*}
&\E\left[  \int_0^t  e^{-2\rho s}\sigma^2 \big| V'(\widehat{X}_s) \widehat\pi_s\big|^2 \d s \right]=\E\left[ \int_0^t  (e^{-\rho s}\sigma a)^2 \widehat{X}_s^{2(1-\gamma)}  \d s \right]<\infty.
\end{align*}
Hence, taking the expectation in \eqref{eq322} the stochastic integral disappears, leading to
\begin{equation}		\label{eq323}
\begin{aligned}
V(x) & =\mathbb{E}\left[e^{-\rho t} V(\widehat{X}_t)\right]+\mathbb{E}\left[\int_0^t e^{-\rho s} \frac{\widehat{c}_s\,^{1-\gamma}}{1-\gamma} \d s\right] .\\
\end{aligned}
\end{equation}
Considering that $V$ is nonpositive, we get
\begin{equation}	\label{eq324}
\begin{aligned}
V(x) & \leq \mathbb{E}\left[\int_0^t e^{-\rho s} \frac{\widehat{c}_s\,^{1-\gamma}}{1-\gamma} \d s\right].
\end{aligned}
\end{equation}
The integrand in the last integral is nonpositive; hence, by monotone convergence we have
$$
V(x)\leq \lim_{t\to \infty}\mathbb{E}\left[\int_0^t e^{-\rho s} \frac{\widehat{c}_s\,^{1-\gamma}}{1-\gamma} \d s\right] = J(\widehat{c}_\cdot).
$$
and  the claim follows.
%
%
\vspace{-.6cm}
\end{enumerate}
\hspace{5cm}\end{proof}

\subsection{Discussion and comments}	\label{remTrCon}

In this subsection, we elaborate on our method relative to the classical approach and discuss the reasons that motivate its use.

Classical verification theorems (see, e.g., \cite[Ch.\,3, Sec.\,5]{PhamBook} or \cite[Ch.\,5, Sec.\,4]{YongZhou}) typically start with a smooth solution \( v \) to the Hamilton-Jacobi-Bellman (HJB) equation, proceeding as follows:
\begin{enumerate}[(i)]
    \item Exploiting the fact that \( v \) is a supersolution to the HJB, one shows that \( v \geq V \);
    \item Exploiting the fact that \( v \) is also a subsolution, and knowing that \( v \geq V \), one concludes that \( v = V \) and simultaneously derives sufficient optimality conditions.
\end{enumerate}
In the finite horizon Merton problem, both steps can usually be carried out without major issues.

However, in the infinite horizon case, both steps  rely on the validity of transversality conditions  as \( t \to \infty \) which must be verified for the specific problem at hand. The typical transversality conditions that ensure the argument's validity are (see, e.g., \cite[Th.\,3.5.3]{PhamBook}):
\begin{enumerate}[(i)]
    \item The limit condition
    \begin{equation}\label{stra}
    \limsup_{t \to \infty} \mathbb{E}\left[ e^{-\rho t} v\left( X_{t}^{x;c_\cdot,\pi_\cdot} \right) \right] \geq 0 \quad \forall (c_\cdot, \pi_\cdot) \in \mathcal{A}(x),
    \end{equation}
    which appears in the first step of the proof.
    \item The limit condition
    \begin{equation}\label{strabis}
    \liminf_{t \to \infty} \mathbb{E}\left[ e^{-\rho t} v( \widehat{X}_{t}^{x;c_\cdot,\pi_\cdot} ) \right] \leq 0,
    \end{equation}
    used in the second step.
\end{enumerate}

In the case of the Merton problem with \( \gamma \in (0,1) \), one can\footnote{See, e.g., \cite[IV.\,Example~5.2]{FS2006} and \cite[Section~3.6.2]{PhamBook}}:
\begin{itemize}
    \item (i) eliminate the first step because the value function is nonnegative in this case;
    \item (ii) perform the second step straightforwardly using the condition \eqref{eq:standing}.
\end{itemize}

Unfortunately, when \(\gamma > 1\), the initial step becomes significantly more challenging, as \eqref{stra} does not hold for all admissible strategies. For instance, consider the feedback strategy defined by \( c(x) = \alpha x \) and \( \pi(x) \equiv 0 \). This strategy fails to satisfy the condition when \(\alpha\) is large, with the candidate value function
\[
v(x) = a \frac{x^{1-\gamma}}{1-\gamma}.
\]
In this scenario, the state process evolves as
$
X_t = x e^{(r - \alpha)t},
$
and consequently,
\[
\lim_{t \to \infty} \mathbb{E}\left[ e^{-\rho t} \frac{X_t^{1-\gamma}}{1-\gamma} \right] = -\infty,
\]
if
\[
\alpha > \frac{\rho}{\gamma - 1} + r.
\]

Our approach is more straightforward: {since we already know that \(V\) is nonzero (see Proposition \ref{VfinitaOm} and  Remark \ref{rm:uniq})} and solves the Hamilton-Jacobi-Bellman (HJB) equation beforehand, we can bypass the first step and focus solely on the second. This is why we refer to it as a ``half-verification''. Moreover, in our framework, verifying the second step does not require checking any transversality condition because we can eliminate the term
\[
\mathbb{E}\left[ e^{-\rho \tau_{n}} V( X^{*}_{\tau_{n}} ) \right]
\]
from the proof, thanks to its favorable sign (see the transition from equation \eqref{eq323} to \eqref{eq324}).

Finally, note that this approach can also be applied in the finite horizon case and in the infinite horizon case when \(\gamma \in (0,1)\). However, while this would simplify the verification process, it would also entail replicating some of the work carried out in Subsections \ref{subsec:finiteness} and \ref{sec:HJB}.


\end{document}